\documentclass[11pt, reqno]{amsart}

\usepackage{comment}
\usepackage[rgb,dvipsnames]{xcolor}
\usepackage{stmaryrd}
\usepackage{amsfonts}
\usepackage{amssymb}
\usepackage{amsmath,amsthm}
\usepackage{latexsym}
\usepackage{amscd}
\usepackage{esint}
\usepackage{mathrsfs}
\usepackage{faktor}
\usepackage{dsfont}
\usepackage{setspace}
\usepackage{enumitem}
\usepackage[margin=1.2in]{geometry}
\usepackage{bm}

\usepackage{hyperref}
\usepackage{comment}
\usepackage{enumitem}

\newtheorem{prop}{Proposition}[section]
\newtheorem{thm}[prop]{Theorem}
\newtheorem{lemm}[prop]{Lemma}

\newtheorem*{claim*}{Claim}

\theoremstyle{definition}
\newtheorem{defi}[prop]{Definition}
\newtheorem{rmk}[prop]{Remark}

\newcommand{\NN}{\mathbb{N}}

\newcommand{\RR}{\mathbb{R}}

\newcommand{\ZZ}{\mathbb{Z}}

\newcommand{\cA}{\mathcal A}
\newcommand{\cB}{\mathcal B}
\newcommand{\cC}{\mathcal C}

\newcommand{\cE}{\mathcal E}

\newcommand{\cK}{\mathcal K}

\newcommand{\cN}{\mathcal N}

\newcommand{\cP}{\mathcal P}

\newcommand{\cT}{\mathcal T}
\newcommand{\cU}{\mathcal U}
\newcommand{\cV}{\mathcal V}

\newcommand{\ep}{\varepsilon}

\newcommand{\Vol}{\text{Vol}}

\newcommand{\pa}[2]{\frac{\partial #1}{\partial #2}}

\newcommand{\paop}[1]{\pa{}{#1}}

\newcommand{\Area}{\text{Area}}

\newcommand{\rom}[1]{\expandafter\romannumeral #1}
\newcommand{\Rom}[1]{\uppercase\expandafter{\romannumeral #1}}

\setlist[enumerate]{leftmargin = 2em}
\allowdisplaybreaks

\numberwithin{equation}{section}

\title[Existence of curves with constant geodesic curvature]{Existence of curves with constant geodesic curvature in a Riemannian 2-sphere}
\author{Da Rong Cheng}
\address{Department of Pure Mathematics, University of Waterloo, Waterloo, ON, N2L 3G1 Canada}
\email{drcheng@uwaterloo.ca}
\author{Xin Zhou}
\address{Department of Mathematics, Cornell University, Ithaca, NY 14853, USA, and Department of Mathematics, University of California Santa Barbara, Santa Barbara, CA 93106, USA}
\email{xinzhou@cornell.edu}

\begin{document}

\begin{abstract} 
We prove the existence of immersed closed curves of constant geodesic curvature in an arbitrary Riemannian 2-sphere for almost every prescribed curvature. To achieve this, we develop a min-max scheme for a weighted length functional.   
\end{abstract}

\maketitle

\section{Introduction}

In this paper, we investigate the existence of closed curves with prescribed constant geodesic curvature in an arbitrary Riemannian 2-sphere. Aside from being natural geometric objects, curves with constant geodesic curvature are also trajectories of charged particles in a special class of magnetic fields, and the existence problem has been studied extensively. Moreover, a famous conjecture by Arnold \cite[1981-9]{Arnold04} asserts that every Riemannian 2-sphere contains at least two distinct closed curves with geodesic curvature $\kappa$ for any $0<\kappa<\infty$. Many approaches have been developed to tackle this conjecture, including Morse-Novikov theory, symplectic topology methods, degree theory, and variational theory. Nevertheless, the conjecture remains open for an arbitrary Riemannian 2-sphere. Our main result, using a new variational scheme analogous to that in our earlier work \cite{Cheng-Zhou20}, establishes the existence of at least one such curve for almost every prescribed geodesic curvature, without any curvature assumption on the ambient metric.

\begin{thm}\label{thm:main1}
Given a 2-sphere with an arbitrary Riemannian metric $(S^2, g)$, for almost every $\kappa>0$, there exists a nontrivial closed immersed curve with constant geodesic curvature $\kappa$.
\end{thm}

\begin{rmk}
We learned after completing this work that Theorem~\ref{thm:main1} is already contained in the work of Asselle and Benedetti~\cite{Asselle-Benedetti16}, where solutions are interpreted as magnetic geodesics. However, we think our method, which involves perturbing the length functional and hence differs from that in~\cite{Asselle-Benedetti16}, should still be of independent interest. (For other related work on magnetic geodesics, see for instance~\cite{Asselle-Benedetti16, Asselle-Schmaschke18, AAMT17} and the references therein.)
\end{rmk}

To write down the differential equation of interest, assume that $(S^2, g)$ is isometrically embedded into some $\RR^N$. The solutions we produce are smooth curves $u: S^1 \to (S^2, g) \subset \RR^N$ satisfying $|u'| = 1$ and 
\begin{equation}\label{eq:cgc-equation}
u'' = A_u(u', u') + \kappa \cdot Q_u(u'),
\end{equation}
where $(\cdot)'$ denotes component-wise differentiation, $\kappa$ is the prescribed geodesic curvature, and $A, Q$ are, respectively, the second fundamental form of the embedding $(S^2, g) \to \RR^N$ and the almost-complex structure on $S^2$ induced by the metric and the orientation. An outline of our proof of Theorem~\ref{thm:main1} is given below following a brief discussion of some related previous results. Here we mention that while Theorem~\ref{thm:main1} yields a solution for almost every $\kappa$, we do not have a uniform control over their lengths as $\kappa$ varies. Conceivably, the method for~\cite[Theorem 1.2]{Cheng-Zhou20} can be used to overcome this and remove the restriction on $\kappa$ in Theorem~\ref{thm:main1} when $(S^2, g)$ has positive curvature, but we do not pursue this here.

\vspace{1em}
We now briefly review previous results on the Arnold conjecture. This is by no means an exhaustive survey, and the reader is encouraged to consult the references and their bibliographies for a more complete picture. Focusing on the multi-valued nature of the relevant functional (see Section~\ref{subsec:prelim} below), Novikov~\cite{Novikov82} proposed a generalization of Morse theory, now known as the Morse-Novikov theory, and as an application it was shown by Novikov-Taimanov \cite{Novikov-Taimanov84} that closed embedded curve of curvature $\kappa$ exists for strong field, that is, for $\kappa$ large; see also \cite{Taimanov91}.  On the other hand, using the symplectic topology methods introduced by Arnold \cite{Arnold86}, Ginzburg (see the survey \cite{Ginzburg96} and references therein) obtained the existence of close curves of constant geodesic curvature $\kappa$ for small and large values of $\kappa$. Recently, Schneider has developed a degree theory for immersed curves in \cite{Schneider11} and confirmed Arnold's conjecture assuming nonnegative Gaussian curvature \cite{Schneider12}; see also \cite{Rosenberg-Smith20, Rosenberg-Schneider11}. Finally, we mention that Zhou-Zhu \cite{Zhou-Zhu20b}, using the CMC min-max theory in the Almgren-Pitts framework developed in \cite{Zhou-Zhu19, Zhou-Zhu20}, proved the existence of networks of constant geodesic curvature $\kappa$ for any $\kappa>0$ on an arbitrary closed Riemannian surface, and Ketover-Liokumovich \cite{Ketover-Liokumovich18} has further shown that the network is $C^{1, 1}$ and has at most 1 node point.

By comparison, the existence of close geodesics on Riemannian 2-spheres, and also on general closed Riemannian manifolds, dates back to Birkhoff \cite{Birkhoff17}. See also \cite{Colding-Minicozzi08a} for a strengthened version by Colding-Minicozzi. Since then, tremendous progress has been made on the existence of closed geodesics, and we refer the reader to \cite{Marques-Neves17} for a nice summary. In particular, every Riemannian 2-sphere contains at least three distinct simple closed geodesics~\cite{Lyusternik-Snirelman47, Grayson89}, and also infinitely many possibly self-intersecting ones~\cite{Franks92, Bangert93}.

\subsection*{Overview of proof}

Formally speaking, the curves we seek are unit-speed critical points of the following weighted length functional:
\[
\int_{S^1} |u'|\ d\theta + \kappa \cdot A(f_u) = : L(u) + \kappa \cdot A(f_u) ,
\]
where $A(f_u)$ is the area enclosed by a choice of extension $f_u$ of $u$ to the unit disk and is only well-defined up to an integer multiple of $\Area_g(S^2)$. (See Section~\ref{subsec:prelim} for more details.) However, this weighted functional is not directly amenable to standard variational techniques due to at least two issues. First of all, the length of a curve is invariant under reparametrizations. Secondly, since it is possible for the enclosed area term to become very negative, the weighted functional is not bounded from below, and nor does it control the length of the curve.

To overcome the first difficulty, instead of replacing the length by the energy as in the classical setting of geodesics, where the enclosed area term is absent, we perturb the functional by replacing $L(u)$ with
\[
L_{\ep}(u) := \int_{S^1} \big[(\ep^2 + |u'|^2)^{\frac{1 + \ep}{2}}  - \ep^{1 + \ep}\big]d\theta,
\]
thereby obtaining a $C^1$-functional whose critical points are smooth with constant speed. We note that a similar regularization was used by Bahri-Taimanov \cite{Bahri-Taimanov98} in finding periodic orbits in magnetic fields. However, in their situation magnetic fields are assumed to be exact, and as a result their functionals are single-valued. Also, the Morse-theoretic arguments in~\cite{Bahri-Taimanov98} require their perturbed functionals to be $C^2$, whereas $C^1$ suffices for our present purposes.

Next, to address the second difficulty and also obtain uniform estimates with respect to the parameter $\ep$, we utilize the monotonicity properties of the min-max value with respect to $\kappa$ along with Fatou's lemma, as done by Struwe~\cite{Struwe88}, to obtain, for almost every $\kappa$, a minimizing sequence of sweepouts whose almost-maximal slices enjoy $L_{\ep}$-bounds independently of $\ep$. (See also our previous work on CMC spheres \cite{Cheng-Zhou20}.) In addition, the non-triviality of the sweepouts together with an isoperimetric inequality allow us to establish a uniform lower bound on $L_{\ep}$ for these slices as well. The upper and lower bounds then put us in a position to apply a standard deformation argument using pseudo-gradient flows to produce non-constant critical points $u_{\ep}$ of the perturbed functionals with $L_{\ep}(u_{\ep})$ uniformly bounded, which implies a uniform Lipschitz estimate as each $u_\ep$ has constant speed. To conclude the proof of Therem~\ref{thm:main1}, we use the Euler-Lagrange equations of $L_{\kappa, \ep}$ to obtain higher-order estimates independent of $\ep$ and pass to a subsequential limit as $\ep \to 0$ to get a smooth, constant-speed curve which is non-trivial thanks to the uniform lower bound on $L_{\ep}(u_{\ep})$, and satisfies~\eqref{eq:cgc-equation} up to reparametrization.

We remark that if we proceed as in the case of geodesics and replace $L(u)$ by
\[
\frac{1}{2}\int_{S^1}|u'|^2 d\theta
\]
when perturbing the weighted length functional, we would obtain in the end a solution to~\eqref{eq:cgc-equation} with constant, but not necessarily unit, speed. The geodesic curvature of the resulting curve is then $\kappa$ divided by this constant speed, which we could not prescribe.

\subsection*{Organization} 
In Section \ref{S:perturbed functional}, we set up some notation and define the perturbed functional $L_{\kappa, \ep}$. We show that critical points of $L_{\kappa, \ep}$ are smooth and prove a version of the Palais-Smale condition for $L_{\kappa, \ep}$. Most of our effort goes into Section \ref{S:existence of critical points}, where we find nontrivial critical points of $L_{\kappa, \ep}$ with $L_\ep$ bounded independent of $\ep$ for almost all prescribed curvatures $\kappa>0$. Two key ingredients are a derivative bound for the min-max values with respect to $\kappa$, contained in Section \ref{SS:derivative estimates}, and a localized version of pseudo-gradient flow argument, contained in Section \ref{SS:pseudo-gradient flow}. The proof of the main result is completed in Section \ref{S:pass to limit} by analyzing the limit of critical points when $\ep \to 0$.

\vspace{1em}
{\bf Acknowledgement}: X. Z. is partially supported by NSF grant DMS-1811293, DMS-1945178, and an Alfred P. Sloan Research Fellowship.

\section{The perturbed functional}\label{S:perturbed functional}
Below, $S^1$ denotes the unit circle, and functions on $S^1$ are identified with $2\pi$-periodic functions on $\RR$. We assume that the target $(S^2, g)$ is isometrically embedded into some $\RR^N$, and denote by $A$ the second fundamental form of this embedding. By $\cV$ we mean a tubular neighborhood of $S^2$ in $\RR^N$ on which the nearest-point projection, denoted $\Pi: \cV \to S^2$, has bounded derivatives of all orders. For brevity, the differential of $\Pi$ will be denoted $P: \cV \to \RR^{N \times N}$. Thus for $y \in \cV$, the matrix $P_y$ represents orthogonal projection onto $T_{\Pi(y)}S^2$. Given a map $v : S^1 \to \cV$, we write $P_v$ for the composition $P \circ v: S^1 \to \RR^{N \times N}$.

The metric $g$ and the volume form on $(S^2, g)$ determines an orthogonal almost-complex structure on $TS^2$, which we denote by $Q$, given by
\[
\Vol_g(X, Y) = g(X, Q(Y)).
\]

\subsection{Preliminaries}\label{subsec:prelim}
We will work with the following Sobolev spaces for $\ep > 0$ small:
\[
W^{1, 1 + \ep}(S^1; S^2) = \{u \in W^{1, 1 + \ep}(S^1; \RR^N)\ |\ u(\theta) \in S^2 \text{ for all }\theta \in S^1 \}.
\]
We equip it with the subspace topology coming from the $W^{1, 1 + \ep}$-norm
\[
\|u\|_{1, 1 +\ep} = \big( \int_{S^1} |u|^{1 + \ep} + |u'|^{1 + \ep} d\theta \big)^{\frac{1}{1 + \ep}}.
\]
The space $W^{1, 1 + \ep}(S^1; S^2)$ is a smooth, closed submanifold of the Banach space $W^{1, 1 + \ep}(S^1; \RR^N)$, with the tangent space at $u$ identified with
\[
\cT_u = \{ \psi \in W^{1, 1 + \ep}(S^1; \RR^N)\ |\ \psi(\theta) \in T_{u(\theta)}S^2 \text{ for all }\theta \in S^1 \},
\] 
which is a closed subspace of $W^{1, 1 + \ep}(S^1; \RR^N)$ with a closed complement. Letting $\Theta_u(\psi) = \Pi(u + \psi)$ for $\psi \in \cT_u$, then for small enough balls $\cB_u$ around the origins in $\cT_u$, the collection $\{ (\cB_u, \Theta_u |_{\cB_u}) \}_{u \in W^{1, 1 + \ep}(S^1; S^2)}$ form a smooth atlas. Restricting the $W^{1, 1 + \ep}$-norm to each tangent space $\cT_u$ yields a Finsler structure on $W^{1, 1 + \ep}(S^1; S^2)$.

For $u \in W^{1, 1 + \ep}(S^1; S^2)$, consider
\[
\cE(u) = \{f \in C^0([0, 1] \times S^1; S^2)\ |\ f(0, \cdot) = \text{ constant, }f(1, \cdot) = u\},
\]
which can be thought of as the set of extensions of $u$ to the unit disk $D$. A key component of the perturbed functionals is the area enclosed by extensions, which we now define. First, denoting by $\theta$ and $t$ the variables on $S^1$ and $[0, 1]$, respectively, we introduce the following more restrictive class of extensions
\[
\widetilde{\cE}(u) = \{ f \in \cE(u) \cap W^{1, 1}([0, 1] \times S^1; S^2)\ |\ f_{\theta} \in L^{1 + \ep}([0, 1] \times S^1),\ f_t \in L^{\frac{1 + \ep}{\ep}}([0, 1] \times S^1)\},
\]
and equip it with the topology coming from the norm
\[
f \mapsto \|f\|_{C^0} + \|f_{t}\|_{\frac{1 + \ep}{\ep}} + \|f_\theta\|_{1 + \ep}.
\]
For $u \in W^{1, 1 + \ep}(S^1; S^2)$ and $f \in \widetilde{\cE}(u)$, we define the area enclosed by $u$ with respect to the extension $f$ by
\[
A(f) = \int_{[0, 1] \times S^1} f^{\ast}\Vol_{g}. 
\]
The basic properties of the enclosed area are summarized below.
\begin{lemm}
\label{lemm:area-properties}
Given  $u \in W^{1, 1 + \ep}(S^1; S^2)$, the following hold.
\begin{enumerate}
\item[(a)] For $f_0, f_1 \in \widetilde{\cE}(u)$ we have 
\[
A(f_0) - A(f_1)  = k \Area_g(S^2) \text{ for some }k \in \ZZ. 
\]
\item[(b)] Let $F:[0, 1] \to \widetilde{\cE}(u)$ be a continuous path. Then $A(F(0)) = A(F(1))$.
\item[(c)] There exists a universal constant $\delta_0$ such that if $f_0, f_1 \in \widetilde{\cE}(u)$ and $\|f_0 - f_1\|_{C^0} < \delta_0$ then $A(f_1) = A(f_0)$.
\end{enumerate}
\end{lemm}
\begin{proof}
For part (a), we consider the concatenation $h = f_0 + (-f_1)$, defined by
\[
h(t, \theta) = \left\{
\begin{array}{cc}
f_0(2t, \theta) & \text{, if }t \in [0, 1/2],\\
f_1(2 -  2t, \theta) & \text{, if } t\in [1/2, 1].
\end{array}
\right.
\]
The map $h$ lies in $C^{0} \cap W^{1, 1}([0, 1] \times S^1; S^2)$ with $h_\theta \in L^{1 + \ep}$ and $h_t \in L^{\frac{1 + \ep}{\ep}}$, and induces a continuous map $\widehat{h}: S^2 \to S^2$. Next, the integrability of $h_\theta$ and $h_t$, along with the fact that $h(0, \cdot)$ and $h(1, \cdot)$ are constant maps, imply that we may approximate $h$ with maps $h_j \in C^{\infty}([0, 1] \times S^1; S^2)$ such that $h_j(t, \cdot) = h(0, \cdot)$ near $t = 0$, $h_j(t, \cdot) = h(1, \cdot)$ near $t  = 1$, and 
\[
\|h_j - h\|_{C^0} + \|(h_j)_{t} - h_t\|_{\frac{1 + \ep}{\ep}} + \|(h_j)_{\theta} - h_{\theta}\|_{1 + \ep} \to 0.
\]
It follows from the existence of such approximations that $\deg(\widehat{h})$ can be computed by\\ $\frac{1}{\Area_g(S^2)}\int_{[0,1] \times S^1} h^{\ast}\Vol_{g} = \frac{A(f_0) - A(f_1)}{\Area_g(S^2)}$, hence the latter is an integer. Part (b) follows from part (a) and the continuity of $t \mapsto A(F(t))$. Part (c) can be deduced by applying (b) to $F(s, \cdot) = \Pi(sf_1 + (1 - s)f_0)$, provided $\delta_0$ is chosen so that the $3\delta_0$-distance neighborhood of $S^2$ is still contained in $\cV$.
\end{proof}

\begin{rmk}\label{rmk:area-extend}
Note that Lemma~\ref{lemm:area-properties}(c) allows us to define the enclosed area functional on $\cE(u)$ by approximation by letting 
\[
A(f) = A(\widetilde{f}),
\]
where $\widetilde{f} \in \widetilde{\cE}(u)$ is such that $\|f - \widetilde{f}\|_{C^0} < \delta_0/2 $. To see that such approximations exist, let $\zeta: \RR \to [0, 1]$ be a smooth cut-off function with $\zeta(t) = 0$ for $t \leq 0$ and $\zeta(t) =1$ for $t \geq 1$. For $f \in \cE(f)$, we first erase its boundary value by considering
\[
\widehat{f}(t, \cdot) = f(t, \cdot) - \big( \zeta(t) f(1, \cdot) + (1 - \zeta(t))f(0, \cdot) \big).
\]
This can be approximated uniformly on $[0, 1] \times S^1$ by maps $\widehat{f_j} \in C^{\infty}([0, 1]\times S^1; \RR^N)$ which vanish near $\{0, 1\} \times S^1$. In particular, eventually it makes sense to define
\[
f_j = \Pi(\widehat{f_j} + \big( \zeta(t) f(1, \cdot) + (1 - \zeta(t))f(0, \cdot) \big)),
\]
which agrees with $f$ on $\{0, 1\} \times S^1$, converges uniformly to $f$, and also lies in $\widetilde{\cE}(u)$, since the map $\big( \zeta(t) f(1, \cdot) + (1 - \zeta(t))f(0, \cdot) \big)$ belongs to $C^{0} \cap W^{1, 1}([0, 1] \times S^1; \RR^N)$ and has $t$-derivative and $\theta$-derivative lying in $L^{\infty}$ and $L^{1 + \ep}$, respectively. 

It is also not hard to see that Lemma~\ref{lemm:area-properties}(a)(c) continue to hold for the extended functional, with $\widetilde{\cE}(u)$ replaced by $\cE(u)$, and that part (b) holds with the $C^{0}$-topology on $\cE(u)$ by partitioning $[0, 1]$ and applying part (c) repeatedly. We omit the details. 
\end{rmk}
Next we define the perturbations of the weighted length functional.
\begin{defi}\label{defi:perturbed-functional}
For $\kappa, \ep > 0$, and $u \in W^{1, 1 + \ep}(S^1; S^2), f \in\cE(u)$, we let
\[L_{\kappa, \ep}(u, f) = L_{\ep}(u) + \kappa \cdot A(f),\]
where
\[
L_{\ep}(u) = \int_{S^1} \big[(\ep^2 + |u'|^2)^{\frac{1 + \ep}{2}}  - \ep^{1 + \ep}\big]d\theta.
\]
\end{defi}
\noindent Here $L_{\ep}$ is a regularization of the length functional $L(u) = \int_{S^1}|u'| d\theta$. We end this preliminary section by collecting some standard estimates that will be used repeatedly. 
\begin{lemm}\label{lemm:basic-estimates}
There exists a universal constant $A_0$ such that the following hold for all $\ep \in (0, 1)$ and $u, v \in W^{1, 1 + \ep}(S^1; S^2)$.
\begin{enumerate}
\item[(a)] $L(u) \leq A_0 \ep + A_0( L_{\ep}(u))^{\frac{1}{1 + \ep}}$.
\vskip 1mm
\item[(b)] $|L_{\ep}(u) - L_{\ep}(v)| \leq A_0\big(  \|u'\|_{1 + \ep} + \|v'\|_{1 + \ep} \big)^{\frac{1 + \ep}{2}}\|u - v\|_{1, 1 + \ep}^{\frac{1 + \ep}{2}}$.
\vskip 1mm
\item[(c)] Suppose $\|u - v\|_{C^0} < \delta_0$ and let $f \in \cE(u)$ be an extension of $u$. Define $h \in \cE(v)$ by concatenating $f$ with $q: (t, \theta) \mapsto \Pi(tv(\theta) + (1- t)u(\theta))$. Then 
\[|A(f) - A(h)| \leq A_0 \|u - v\|_{C^0}\big( L(u) + L(v) \big).\]
\end{enumerate}
\end{lemm}
\begin{proof}
For part (a) we first use H\"older's inequality to get $L(u) \leq C\big( L_{\ep}(u) + \ep^{1 + \ep} \big)^{\frac{1}{1 + \ep}}$. The conclusion then follows from the fact that for $\alpha \in (0, 1)$ we have
\begin{equation}\label{eq:fake-triangle}
(t + s)^{\alpha} \leq t^{\alpha} + s^{\alpha} \text{ for all }t, s \geq 0.
\end{equation}
For part (b), note that the above inequality implies, still for $\alpha \in (0, 1)$, that
\begin{equation}\label{eq:fake-triangle-2}
|t^\alpha - s^\alpha| \leq |t - s|^{\alpha} \text{ for all }t, s \geq 0.
\end{equation}
Applying this with $\alpha = \frac{1 + \ep}{2}$ (recall that $\ep \in (0, 1)$ by assumption), we obtain
\begin{align*}
|L_{\ep}(u) - L_{\ep}(v)| &\leq \int_{S^1}\big| (\ep^2 + |u'|^2)^{\frac{1 + \ep}{2}} -  (\ep^2 + |v'|^2)^{\frac{1 + \ep}{2}} \big| d\theta\\
&\leq \int_{S^1} \big||u'|^2 - |v'|^2 \big|^{\frac{1 + \ep}{2}} d\theta =  \int_{S^1} \big||u'| +  |v'| \big|^{\frac{1 + \ep}{2}} |u'  - v'|^{\frac{1 + \ep}{2}} d\theta \\
&\leq \big( \int_{S^1} (|u'| + |v'|)^{1 + \ep} \big)^{\frac{1}{2}} \big( \int_{S^1}|u' - v'|^{1 + \ep} \big)^{\frac{1}{2}},
\end{align*}
which implies the result. For part (c), since $q \in C^0 \cap W^{1, 1}([0, 1] \times S^1; S^2)$ with $q_\theta \in L^{1 + \ep}([0, 1] \times S^1; S^2)$ and $q_{t} \in L^{\infty}([0, 1] \times S^1; S^2)$, it's not hard to see that $A(h) - A(f) = \int_{[0, 1]\times S^1} q^{\ast}\Vol_{g}$. That is,
\begin{align*}
|A(f) - A(h)| &\leq \int_{0}^1 \int_{S^1} |(\Vol_g)_{q(t, \theta)}(P_{tv + (1 - t)u}(v -u), P_{tv + (1- t)u}(tv' + (1 - t)u'))| d\theta dt\\
& \leq C \int_{S^1} |v - u|(|u'| + |v'|) d\theta,
\end{align*}
this clearly gives the desired estimate.
\end{proof}

To state the next lemma, given $\ep \in (0, 1)$, for convenience we define $F:\RR^{N} \to \RR$ by 
\[
F(y) = (\ep^2 + |y|^2)^{\frac{1 + \ep}{2}}.
\]
Of course the derivative of $F$ is given by $(dF)_y = (1 + \ep)\frac{y}{(\ep^2 + |y|^2)^{\frac{1 - \ep}{2}}}$. We collect two standard estimates below for later use. The proofs are included in Appendix \ref{sec:standard-estimates} for the sake of completeness.
\begin{lemm}\label{lemm:coercive} 
For all $y_0, y_1 \in \RR^N$, there hold
\begin{equation}\label{eq:F-estimate1}
\big((dF)_{y_1} - (dF)_{y_0}\big) \cdot (y_1 - y_0) \geq c_{\ep}\frac{|y_1 - y_0|^2}{(\ep^2 + |y_1|^2 + |y_0|^2)^{\frac{1 - \ep}{2}}}.
\end{equation}
\begin{equation}\label{eq:F-estimate2}
|(dF)_{y_1} - (dF)_{y_0}| \leq C(\ep^2 + |y_1|^2 + |y_0|^2)^{\ep/4} |y_1 - y_0|^{\ep/2}.
\end{equation}
\end{lemm}

\subsection{Local reduction}\label{subsec:local-reduct}
The fact that $L_{\kappa, \ep}$ depends both on the map $u$ and the extension $f$ is a major technical point that we need to keep track of throughout the paper. Fortunately, we may eliminate the $f$-dependence locally on simply-connected neighborhoods. Specifically, given a simply-connected open set $\cA \subset W^{1, 1 + \ep}(S^1; S^2)$ and a map $u_0 \in \cA$, along with an extension $f_0 \in \cE(u_0)$, then for any other $u \in \cA$ we have by connectedness a path $h: [0, 1] \to \cA$ leading from $u_0$ to $u$. Concatenating $f_0$ with the map $(t, \theta) \mapsto h(t, \theta)$ yields an extension $f_u \in \cE(u)$, and we define
\[
L^{\cA}_{\kappa, \ep}(u) = L_{\kappa, \ep}(u, f_u).
\]
We now verify that $L^{\cA}_{\kappa, \ep}$ is well-defined, $C^{1}$-functional on $\cA$.
\begin{lemm}\label{lemm:local-reduct}
\begin{enumerate}
\item[(a)] $L^{\cA}_{\kappa, \ep}$ is well-defined. That is, the choice of path $h$ is irrelevant.
\item[(b)] $L^{\cA}_{\kappa, \ep}$ is a $C^{1}$-functional on $\cA$.
\end{enumerate}
\end{lemm}
\begin{proof}
Part (a) follows from Lemma~\ref{lemm:area-properties}(b) and simply-connectedness. Next, take $u \in \cA$ and consider a chart $(\cB_u, \Theta_u)$ centered at $u$. Below we drop the subscripts in $\cB_u, \Theta_u$ for brevity. To prove (b), it suffices to show that $L_{\kappa, \ep}^{\cA} \circ \Theta$ is $C^1$ on $\cB$. To that end, note that for $\psi \in \cB$, letting $\widetilde{f}(t, \theta) = \Theta(t\psi(\theta))$, we have
\begin{equation}\label{eq:difference-for-C1}
L_{\kappa, \ep}^{\cA}(\Theta (\psi)) = \int_{S^1} (\ep^2 + |\Theta(\psi)'|^2)^{\frac{1 + \ep}{2}}d\theta + \kappa\int_{[0, 1]\times S^1}\widetilde{f}^{\ast}\Vol_g + \kappa A(f_u) - 2\pi \ep^{1 + \ep}.
\end{equation}
The map $\psi \to \Theta(\psi)$ is in fact smooth from $\cB$ to $W^{1, 1 + \ep}(S^1; \RR^N)$. On the other hand, using~\eqref{eq:F-estimate2} from Lemma~\ref{lemm:coercive}, it is not hard to see that 
\[
v \mapsto \int_{S^1}(\ep^2 + |v'|^2)^{\frac{1 + \ep}{2}} d\theta
\]
defines a $C^1$-functional $W^{1, 1 + \ep}(S^1; \RR^N) \to \RR$. Hence the first term on the right-hand side of~\eqref{eq:difference-for-C1} is $C^1$. As for the second term, we write $\{\paop{y^i}\}$ for the coordinates in $\RR^N$, and introduce the functions $a_{ij}: \cV \to \RR$ defined by
\[
a_{ij}(y) = (\Vol_g)_y(P_y(\paop{y^i}), P_y(\paop{y^j})).
\]
Note that these are smooth functions. Moreover, since we may write
\[
\int_{[0, 1]\times S^1}\widetilde{f}^{\ast}\Vol_g = \int_{[0, 1] \times S^1} a_{ij}(u + t\psi) (u^j_\theta + t\psi^j_\theta) \psi^{i} d\theta dt,
\]
it is not hard to see using the smoothness of $a_{ij}$ and the embedding $W^{1, 1 + \ep} \to C^0$ along with H\"older's inequality that the second term on the right-hand side of~\eqref{eq:difference-for-C1} is $C^1$ as well.
\end{proof}

\begin{defi}\label{defi:local-reduct}
The functional $L^{\cA}_{\kappa, \ep}$ is called the \textit{local reduction} of $L_{\kappa, \ep}$ on $\cA$ induced by $u$ and $f$. Note that we are suppressing from the notation its dependence on $u$ and $f$, as these should always be clear from the context. Note also that, by Lemma~\ref{lemm:area-properties}(a) and Lemma~\ref{lemm:local-reduct}(b), any two local reductions differ by a fixed integer multiple of $\kappa \Vol_g(S^2)$ on any connected subset of their common domain.
\end{defi}

\subsection{First variation}
In this section we compute the first variation of $L_{\kappa, \ep}$. We do so with the help of local reductions, and then show that the choice of reduction is irrelevant, and consequently the first variation makes sense globally. To begin, let $\cA \subset W^{1, 1 + \ep}(S^1; S^2)$ be a simply-connected open set on which a local reduction $L^{\cA}_{\kappa, \ep}$ is defined. Since $L^{\cA}_{\kappa, \ep}$ is a $C^1$-functional on $\cA$, at each $u \in \cA$ it has a differential, denoted by $\delta L^{\cA}_{\kappa, \ep}(u)$, which is a bounded linear functional on $\cT_u$, and can be computed by
\begin{equation}\label{eq:first-var-reduct}
\delta L^{\cA}_{\kappa, \ep}(u)(\psi) = \frac{d}{dt}\big|_{t =0} L^{\cA}_{\kappa, \ep}(\Pi(u + t\psi)).
\end{equation}
\begin{defi}
For $u \in W^{1, 1 + \ep}(S^1; S^2)$ we define $\delta L_{\kappa, \ep}(u): \cT_u \to \RR$ by letting 
\[
\delta L_{\kappa, \ep}(u) = \delta L^{\cA}_{\kappa, \ep}(u),
\]
where $L^{\cA}_{\kappa, \ep}$ is any local reduction on a simply-connected neighborhood $\cA$ of $u$. Note that such a neighborhood always exists since $W^{1, 1 + \ep}(S^1; S^2)$ is a manifold. Also, in view of the last remark in Definition~\ref{defi:local-reduct} and the equation~\eqref{eq:first-var-reduct}, $\delta L_{\kappa, \ep}(u)$ is well-defined. 
\end{defi}
\begin{defi}
A map $u \in W^{1, 1 + \ep}(S^1; S^2)$ is a critical point of $L_{\kappa, \ep}$ if $\delta L_{\kappa, \ep}(u) = 0$. 
\end{defi}

To compute $\delta L_{\kappa, \ep}(u)(\psi)$ for $u \in W^{1, 1 + \ep}(S^1; S^2)$ and $\psi \in \cT_u$, we fix $f \in \cE(u)$, consider the local reduction induced by $(u, f)$ on a simply-connected neighborhood $\cA$ of $u$, and carry out the differentiation in~\eqref{eq:first-var-reduct}. Note that for sufficiently small $t$, letting $\widetilde{f}(s, \theta) = \Pi(u + s\psi)$ for $(s, \theta) \in [0, t] \times S^1$,  we have 
\begin{align}
&L^{\cA}_{\kappa, \ep}(\Pi(u + t\psi)) - L^{\cA}_{\kappa, \ep}(u) \nonumber \\
=\ & L_{\ep}(\Pi(u + t\psi)) - L_{\ep}(u) +\kappa \cdot \int_{0}^t \big[ \int_{S^1} (\Vol_g)_{\widetilde{f}}(\widetilde{f}_s, \widetilde{f}_\theta) d\theta\big] ds.\label{eq:quotient-for-first-var}
\end{align}
The $t$-derivative of the integral term at $t = 0$ is equal to
\begin{equation}\label{eq:area-derivative}
\kappa \int_{S^1} (\Vol_g)_{u} \big( P_u(\psi), u' \big) d\theta = \kappa\int_{S^1} \psi \cdot Q_u(u') d\theta,
\end{equation}
where we used the fact that $(\Vol_g)_u(X, Y) = X \cdot Q_u(Y)$ and $P_u(\psi) = \psi$ to get the equality. On the other hand, for the terms involving $L_\ep$ on the right-hand side of~\eqref{eq:quotient-for-first-var}, we have 
\begin{align}
\frac{d}{dt}\Big|_{t = 0}L_\ep (\Pi(u + t\psi)) 
&= (1 + \ep) \int_{S^1} (\ep^2 + |u'|^2 )^{\frac{\ep - 1}{2}} u' \cdot \psi' d\theta. \label{eq:L-derivative}
\end{align}
Here no projection is required since $\psi \in \cT_u$. Putting together~\eqref{eq:area-derivative} and~\eqref{eq:L-derivative}, we obtain the following first variation formula for $L_{\kappa, \ep}$:
\begin{equation}\label{eq:first-var}
\delta L_{\kappa, \ep}(u)(\psi) = \int_{S^1} (1 + \ep)\frac{ u' \cdot \psi' }{(\ep^2 + |u'|^2)^{\frac{1 - \ep}{2}}} + \kappa \psi \cdot Q_u(u')  d\theta.
\end{equation}
The norm of $\delta L_{\kappa, \ep}(u)$, which we denote by $\|\delta L_{\kappa, \ep}(u)\|$, is defined by duality.

It will be convenient for us to have a version of the first variation formula where the test function $\psi$ is allowed to vary in $W^{1, 1 + \ep}(S^1; \RR^N)$. To that end we also define 
\[
G_{\kappa, \ep}(u) = \delta L_{\kappa, \ep}(u) \circ P_u.
\]
We define the norm of $G_{\kappa, \ep}(u)$ by 
\[
\| G_{\kappa, \ep}(u) \| = \sup \{ G_{\kappa, \ep}(u)(\psi)\ |\ \psi \in W^{1, 1 + \ep}(S^1; \RR^N), \|\psi\|_{1, 1 + \ep} \leq 1 \}.
\]
The following lemma establishes a relationship between $\|G_{\kappa, \ep}(u)\|$ and $\| \delta L_{\kappa, \ep}(u)\|$. For the reader's convenience we include a proof in Appendix \ref{sec:standard-estimates}.
\begin{lemm}\label{lemm:G-dL-relation} 
There exists a universal constant $A_1$ such that for $\kappa > 0$ and $\ep \in (0, 1/2)$, we have 
\[\|\delta L_{\kappa, \ep}(u)\| \leq \|G_{\kappa, \ep}(u)\| \leq A_1 (1 + \|u\|_{1, 1+ \ep}) \|\delta L_{\kappa, \ep}(u)\|.
\]
for all $u \in W^{1, 1 + \ep}(S^1; S^2)$.
\end{lemm}

The proposition below gives an alternative expression for $G_{\kappa, \ep}(u)$ which will be useful later, particularly in the next section.
\begin{prop}
\begin{equation}\label{eq:first-var-proj}
G_{\kappa, \ep}(u)(\psi) = \int_{S^1} (1 + \ep) \frac{ u' \cdot \psi' }{(\ep^2 + |u'|^2)^{\frac{1 - \ep}{2}}} + (1 + \ep) \frac{A_u(u', u')}{(\ep^2 + |u'|^2)^{\frac{1 - \ep}{2}}} \cdot \psi + \kappa \psi \cdot Q_u(u').
\end{equation}
\end{prop}
\begin{proof}
Assuming for the moment that $u \in C^{\infty}(S^1; S^2)$, then we have by~\eqref{eq:first-var} that 
\begin{equation}\label{eq:first-var-smooth}
G_{\kappa, \ep}(u)(\psi) = \int_{S^1} \Big[-(1 + \ep) P_u\Big(\big( (\ep^2 + |u'|^2)^{\frac{\ep - 1}{2}}u'\big)'\Big) + \kappa Q_u(u')\Big] \cdot \psi d\theta.
\end{equation}
Now note that of course $u' \in T_u S^2$, while the orthogonal projection of $u''$ onto $(T_uS^2)^{\perp}$ is exactly $A_u(u', u')$. Hence
\[
P_u\Big(\big( (\ep^2 + |u'|^2)^{\frac{\ep - 1}{2}}u'\big)'\Big) = \big( (\ep^2 + |u'|^2)^{\frac{\ep - 1}{2}}u'\big)'  - (\ep^2 + |u'|^2)^{\frac{\ep - 1}{2}} A_u(u', u').
\]
Substituting this into~\eqref{eq:first-var-smooth} and integrating by parts give~\eqref{eq:first-var-proj} when $u \in C^{\infty}(S^1; S^2)$. The case $u \in W^{1, 1 + \ep}(S^1; S^2)$ follows by approximation.
\end{proof}
\subsection{Regularity of critical points and the Palais-Smale condition}
We first establish the regularity of critical points of the perturbed functional.
\begin{prop}\label{prop:smooth}
Let $u \in W^{1, 1 + \ep}(S^1; S^2)$ be a critical point of $L_{\kappa, \ep}$. Then $u$ is smooth with estimates, and moreover $|u'|$ is constant.
\end{prop}
\begin{proof}
Introducing
\[
h = \frac{u'}{(\ep^2 + |u'|^2)^{\frac{1 - \ep}{2}}},
\]
we see that $|h|^2 = (\ep^2 + |u'|^2)^{\ep - 1}|u'|^2$, and hence $|u'|^2 = \tau^{-1}(|h|^2)$, where $\tau$ denotes the (strictly increasing) function $t \mapsto (\ep^2 + t)^{\ep- 1}t$ for $t > -\ep^2$. (Note: $\tau'(t) = (\ep^2 + t)^{\ep - 2}(\ep^2 + \ep t)$.) Consequently 
\begin{equation}\label{eq:u-h-relation}
u' = \big (\ep^2 + \tau^{-1}(|h|^2)\big)^{\frac{1 - \ep}{2}}h.
\end{equation}
From this it follows that $u'$ has the same regularity as $h$, since $\tau^{-1}$ is smooth. Now, the fact that $u$ is a critical point means that 
\begin{equation}\label{eq:Euler-Lagrange}
\int_{S^1} (1 + \ep) \frac{ u' \cdot \psi' }{(\ep^2 + |u'|^2)^{\frac{1 - \ep}{2}}} + (1 + \ep) \frac{A_u(u', u')}{(\ep^2 + |u'|^2)^{\frac{1 - \ep}{2}}} \cdot \psi + \kappa \psi \cdot Q_u(u') = 0,
\end{equation}
for all $\psi \in W^{1, 1 + \ep}(S^1; \RR^N)$. Rearranging, we see from the above that 
\begin{equation}\label{eq:h-ODE}
\int_{S^1}  \langle h, \psi' \rangle = -\int_{S^1} w \cdot \psi,
\end{equation}
where
\[
w :=  \frac{A_u(u', u')}{(\ep^2 + |u'|^2)^{\frac{1 - \ep}{2}}} + (1 + \ep)^{-1} \kappa \cdot Q_u(u').
\]
The smoothness of $h$ follows inductively from~\eqref{eq:h-ODE} and~\eqref{eq:u-h-relation}. To see that $u$ has constant speed, we note from~\eqref{eq:h-ODE} that 
\[
h' = w \perp h,
\]
where the orthogonality follows since $h$ is a multiple of $u'$. Consequently $|h|^2$ is constant, and hence, by~\eqref{eq:u-h-relation}, we see that $|u'|^2$ is constant.
\end{proof}

\begin{prop} \label{prop:PS}Let $u_j$ be a sequence in $W^{1, 1 + \ep}(S^1; S^2)$ satisfying
\begin{enumerate}
\item[(i)] $L_{\ep}(u_j) \leq C$ for some $C$ independent of $j$.
\item[(ii)] $\lim_{j \to \infty}\|\delta L_{\kappa, \ep}(u_j)\| = 0$.
\end{enumerate}
Then a subsequence of $u_j$, which we do not relabel, converges strongly in $W^{1, 1 + \ep}$ to a critical point $u$ of $L_{\kappa, \ep}$ with $L_{\ep}(u) \leq C$.
\end{prop}
\begin{proof}
To begin, we note that assumption (ii) implies that $\lim_{j \to \infty}\|G_{\kappa, \ep}(u_j)\| = 0$ by Lemma~\ref{lemm:G-dL-relation}.  Next, assumption (i) implies that the sequence $u_j$ is bounded in $W^{1, 1 + \ep}$ and hence, passing to a subsequence if necessary, we may assume that $u_j $ converges weakly in $W^{1, 1 + \ep}$ and strongly in $C^{0}$ to some limit map $u \in W^{1, 1 +\ep}(S^1; S^2)$. Next, for $j, k$ large we write
\[
R_{jk} = G_{\kappa, \ep}(u_j)(u_j - u_k) - G_{\kappa, \ep}(u_k)(u_j - u_k).
\]
Then by assumption (ii) and the boundedness of $(u_j)$ in $W^{1, 1 + \ep}$ we see that $|R_{jk}| \to 0$ as $j, k \to \infty$. On the other hand, using~\eqref{eq:first-var-proj} we find that 
\begin{align*}
& \int_{S^1} ((dF)_{u_j'} - (dF)_{u_k'}) \cdot (u_j' - u_{k}') \\
&  =\  R_{jk} - (1 + \ep)\int_{S^1}\big( \frac{A_{u_j}(u_j', u_j')}{(\ep^2 + |u_j'|^2)^{\frac{1 - \ep}{2}}}  - \frac{A_{u_k}(u_k', u_k')}{(\ep^2 + |u_k'|^2)^{\frac{1 - \ep}{2}}}\big) \cdot (u_j - u_k)\\
& - \kappa \int_{S^1} \big(Q_{u_j}(u_j') - Q_{u_k}(u_k')\big) \cdot (u_j - u_k).
\end{align*}
Hence by the convergence properties of $(u_j)$ arranged above, together with~\eqref{eq:F-estimate1}, we see that 
\begin{equation}\label{eq:coercive}
c_{\ep}\int_{S^1}\frac{|u_j' - u_k'|^2}{(\ep^2 + |u_j'|^2+ |u_k'|^2)^{\frac{1 - \ep}{2}}} = o_{j, k}(1),
\end{equation}
where $o_{j, k}(1)$ denotes quantities that tend to zero as $j,k \to \infty$. To continue, we note by H\"older's inequality that 
\begin{align*}
\int_{S^1}|u_j' - u_k'|^{1 + \ep} & = \int_{S^1} |u_j' - u_k'|^{1 + \ep} (\ep^2 + |u_j'|^2 + |u_k'|^2)^{\frac{\ep^2 - 1}{4}}(\ep^2 + |u_j'|^2 + |u_k'|^2)^{\frac{1 - \ep^2}{4}} \\
&\leq \left(\int_{S^1}|u_j' - u_k'|^{2}(\ep^2 + |u_j'|^2 + |u_k'|^2)^{\frac{\ep - 1}{2}} \right)^{\frac{1 + \ep}{2}} \left( \int_{S^1}(\ep^2 + |u_j'|^2 + |u_k'|^2)^{\frac{1 + \ep}{2}} \right)^{\frac{1 - \ep}{2}}.
\end{align*}
This implies by~\eqref{eq:coercive} and the $W^{1, 1 + \ep}$-boundedness of the sequence $u_j$ that $\lim_{j, k \to \infty}\|u_j '- u_k'\|_{1 + \ep} = 0$. In other words, the convergence of $u_j$ to $u$ is strong in $W^{1, 1 + \ep}$. This proves the first conclusion of the Proposition. The second conclusion is obvious.
\end{proof}

\section{Existence of non-trivial critical points}\label{S:existence of critical points}
To any continuous path $\gamma \in C^{0}([0, 1]; W^{1, 1  + \ep})(S^1; S^2)$ starting and ending at constant maps, the map $(t, \theta) \mapsto \gamma(t)(\theta)$ is continuous by Sobolev embedding, and induces a continuous map $h_{\gamma}: S^2 \to S^2$.  We then define the class of admissible sweepouts to be
\[
\cP_{\ep} = \{\gamma \in C^{0}([0,1]; W^{1, 1 + \ep}(S^1; S^2))\ |\ \gamma(0), \gamma(1) = \text{constant},\ \deg(h_{\gamma})  = 1\}.
\]
Given $\gamma \in \cP_{\ep}$ and $t \in [0, 1]$, we define $f_{\gamma, t} \in \cE(\gamma(t))$ by letting $f_{\gamma, t}(s, \theta) = \gamma(st)(\theta)$. For $\kappa, \ep > 0$, the min-max value is defined by
\[
\omega_{\kappa, \ep} = \inf_{\gamma \in \cP_{\ep}}\sup_{t \in [0, 1]} L_{\kappa, \ep}(\gamma(t), f_{\gamma, t}).
\]
We summarize some basic properties of admissible sweepouts and the min-max values below.
\begin{lemm}\label{lemm:min-max-properties}
For $\kappa > 0, \ep > 0$, the following hold.
\begin{enumerate}
\item[(a)] The collection $\cP_{\ep}$ is non-empty, and for each $\gamma \in \cP_\ep$, the function $t \mapsto L_{\kappa, \ep}(\gamma(t), f_{\gamma, t})$ is continuous.
\vskip 1mm
\item[(b)] $0 \leq \omega_{\kappa, \ep} < \infty$.
\end{enumerate}
\end{lemm}
\begin{proof}
For the first assertion of (a), we obtain an element of $\cP_\ep$ by parametrizing appropriately the circles $\{(x_1, x_2, 2t-1)\ |\ x_1^2 + x_2^2 = 1 - (2t-1)^2\}$ for $t \in [0, 1]$ coming from the standard embedding of $S^2$ into $\RR^3$. Next, note that for $\gamma \in \cP_\ep$ and $t_0 \in [0, 1]$, letting $\cA$ be a simply-connected neighborhood of $\gamma(t_0)$ and $L^{\cA}_{\kappa, \ep}$ the local reduction induced by $(\gamma(t_0), f_{\gamma, t_0})$, then for $t$ sufficiently close to $t_0$ we have 
\[
L^{\cA}_{\kappa, \ep}(\gamma(t)) = L_{\kappa, \ep}(\gamma(t), f_{\gamma, t}).
\]
The second assertion of (a) then follows from Lemma~\ref{lemm:local-reduct}(b). For part (b), the finiteness of $\omega_{\kappa, \ep}$ follows easily from part (a). On the other hand, since $L_{\kappa, \ep}(\gamma(0), f_{\gamma, 0}) = 0$ for all $\gamma \in \cP_\ep$, we see that $\omega_{\kappa, \ep} \geq 0$.
\end{proof}

\subsection{Derivative estimates on the min-max value and uniform length bound}
\label{SS:derivative estimates}
\begin{prop}\label{prop:struwe-trick} 
\begin{enumerate}
\item[(a)]  Given $\ep > 0$, the function $\kappa \mapsto \omega_{\kappa, \ep}/\kappa$ is non-increasing.
\item[(b)] Given $\kappa > 0$, we have $\omega_{\kappa, \ep} \leq \omega_{\kappa, 1} + 2\pi$ for all $\ep \in (0, 1)$.
\item[(c)] Given a sequence $\ep_j \to 0$, for almost every $\kappa_0 > 0$, there exist a constant $c$, and a subsequence of $\ep_j$, which we do not relabel, such that 
\[
0 \leq \frac{d}{d\kappa}\Big|_{\kappa = \kappa_0}\left( -\frac{\omega_{\kappa, \ep_j}}{\kappa}\right) \leq c, \text{ for all }j.
\]
\end{enumerate}
\end{prop}
\begin{proof}
For part (a), we note that for $\kappa > \kappa' > 0$, any $u \in W^{1, 1 + \ep}(S^1; S^2)$ and any $f \in \cE(u)$, there holds
\begin{equation}\label{eq:struwe-identity}
\frac{L_{\kappa', \ep}(u, f)}{\kappa'} - \frac{L_{\kappa, \ep}(u, f)}{\kappa} = \frac{\kappa - \kappa'}{\kappa' \cdot \kappa}L_{\ep}(u) \geq 0.
\end{equation}
Now, given $\delta > 0$, we may choose $\gamma \in \cP_{\ep}$ such that 
\[
\max_{t \in [0, 1]} L_{\kappa', \ep}(\gamma(t), f_{\gamma, t}) < \omega_{\kappa', \ep} + \delta.
\]
Dividing by $\kappa'$ and using~\eqref{eq:struwe-identity}, we deduce that 
\[
\frac{\omega_{\kappa, \ep}}{\kappa} \leq \max_{t \in [0, 1]}\frac{L_{\kappa, \ep}(\gamma(t), f_{\gamma, t})}{\kappa} \leq \max_{t \in [0, 1]}\frac{L_{\kappa', \ep}(\gamma(t), f_{\gamma, t})}{\kappa'} < \frac{\omega_{\kappa', \ep}}{\kappa'} + \frac{\delta}{\kappa'},
\]
where the first inequality follows from the definition of $\omega_{\kappa, \ep}$. Since $\delta > 0$ is arbitrary, we get the desired monotonicity.

For part (b) we begin by noting that for all $u \in W^{1, 2}(S^1; S^2)$ and $\ep \in (0, 1)$ we have
\begin{equation}\label{eq:e-upperbound}
\int_{S^1}\big[(\ep^2 + |u'|^2)^{\frac{1 + \ep}{2}} - \ep^{1 + \ep}\big] d\theta \leq \int_{S^1}(1 + |u'|^2)^{\frac{1 + \ep}{2}}d\theta \leq \int_{S^1}\big((1 + |u'|^2) - 1\big) d\theta + 2\pi.
\end{equation}
Next, given $\delta > 0$, we may choose $\gamma \in \cP_1$ such that 
\[
\max_{t \in [0, 1]}L_{\kappa, 1}(\gamma(t), f_{\gamma, t}) < \omega_{\kappa, 1} + \delta.
\]
Combining this with~\eqref{eq:e-upperbound} gives
\[
\omega_{\kappa, \ep} \leq \max_{t \in [0, 1]}L_{\kappa, \ep}(\gamma(t), f_{\gamma}) \leq \max_{t \in [0, 1]}L_{\kappa, 1}(\gamma(t), f_{\gamma, t}) + 2\pi < \omega_{\kappa, 1} + 2\pi + \delta,
\]
where the first inequality follows from the definition of $\omega_{\kappa, \ep}$ and the fact that $\cP_1 \subset \cP_\ep$. Letting $\delta \to 0$ yields the asserted inequality.

For part (c), thanks to part (a) and basic real analysis, we know that the derivative $\frac{d}{d\kappa}\left( -\frac{\omega_{\kappa, \ep_j}}{\kappa}\right)$ exists for almost every $\kappa > 0$ and is non-negative. Moreover, for all $0 < a < b$, there holds
\[
\int_{a}^b\frac{d}{d\kappa}\left( -\frac{\omega_{\kappa, \ep_j}}{\kappa}\right) d\kappa \leq \frac{\omega_{a, \ep_j}}{a} - \frac{\omega_{b, \ep_j}}{b} \leq  \frac{\omega_{a, \ep_j}}{a},
\]
where the last inequality follows from Lemma~\ref{lemm:min-max-properties}(b). Thus, by Fatou's lemma, we obtain
\[
\int_{a}^b \liminf_{j \to \infty}\frac{d}{d\kappa}\left( -\frac{\omega_{\kappa, \ep_j}}{\kappa}\right) d\kappa \leq \liminf_{j \to \infty}\int_{a}^b\frac{d}{d\kappa}\left( -\frac{\omega_{\kappa, \ep_j}}{\kappa}\right) d\kappa  \leq \liminf_{j \to \infty} \frac{\omega_{a, \ep_j}}{a} \leq \frac{\omega_{a, 1} + 2\pi}{a},
\]
where we used part (b) to get the last inequality. Consequently, we have 
\[0 \leq \liminf_{j \to \infty}\frac{d}{d\kappa}\left( -\frac{\omega_{\kappa, \ep_j}}{\kappa}\right)  < \infty, \text{ for almost every $\kappa \in (a, b)$.}\]
The result follows from the arbitrariness of $a$ and $b$.
\end{proof}

We next explain how the derivative estimate in Proposition~\ref{prop:struwe-trick} translates into uniform length upper bounds.
\begin{prop}\label{prop:length-bound}
Suppose for some $\kappa > 0$ and $\ep \in (0, 1/2)$ we have
\[
\frac{d}{d\kappa}\big( -\frac{\omega_{\kappa, \ep}}{\kappa} \big) \leq c,
\]
then for large enough $n$ there exist sweepouts $\gamma_n \in \cP_{\ep}$ with the following properties:
\begin{enumerate}
\item[(a)] $\max_{t \in [0, 1]}L_{\kappa, \ep}(\gamma_n(t), f_{\gamma_n, t}) \leq \omega_{\kappa, \ep} + \frac{\kappa}{n}$.
\vskip 1mm
\item[(b)] $L_{\ep}(\gamma_n(t)) \leq 8\kappa^2c$ whenever $L_{\kappa, \ep}(\gamma_n(t), f_{\gamma_n, t}) \geq \omega_{\kappa, \ep} - \frac{\kappa}{n}$.
\end{enumerate}
\end{prop}
\begin{proof}
We define $\kappa_n = \kappa - (4cn)^{-1}$ and note that by the assumed derivative bound, for large enough $n$ there holds
\begin{equation}\label{eq:difference-quotient}
\frac{1}{\kappa - \kappa_n}\Big( \frac{\omega_{\kappa_n, \ep}}{\kappa_n} - \frac{\omega_{\kappa, \ep}}{\kappa} \Big) \leq 2c.
\end{equation}
By the definition of $\omega_{\kappa_n, \ep}$, we may choose $\gamma_n \in \cP_{\ep}$ such that 
\begin{equation}\label{eq:nearby-min-seq}
\max_{t \in [0, 1]} L_{\kappa_n, \ep}(\gamma_n(t), f_{\gamma_n, t}) < \omega_{\kappa_n, \ep} + \frac{\kappa_n}{2n}.
\end{equation}
Then from~\eqref{eq:struwe-identity} and~\eqref{eq:difference-quotient} we get
\[
\frac{1}{\kappa}\max_{t \in [0, 1]} L_{\kappa, \ep}(\gamma_n(t), f_{\gamma_n, t}) \leq \frac{1}{\kappa_n}\max_{t \in [0, 1]} L_{\kappa_n, \ep}(\gamma_n(t), f_{\gamma_n, t}) < \frac{\omega_{\kappa_n, \ep}}{\kappa_n} + \frac{1}{2n} \leq \frac{\omega_{\kappa, \ep}}{\kappa} + \frac{1}{n}.
\]
This proves property (a). To check (b), suppose for some $t$ we have
\[
\frac{1}{\kappa}L_{\kappa, \ep}(\gamma_n(t), f_{\gamma_n, t}) \geq \frac{\omega_{\kappa, \ep}}{\kappa} - \frac{1}{n}.
\]
Combining this with~\eqref{eq:nearby-min-seq} and~\eqref{eq:struwe-identity}, 
\begin{align*}
\frac{1}{\kappa_n \cdot \kappa} L_{\ep}(\gamma_n(t)) &\leq \frac{1}{\kappa - \kappa_n}\Big( \frac{\omega_{\kappa_n, \ep}}{\kappa_n} - \frac{\omega_{\kappa, n}}{\kappa} + \frac{3}{2n} \Big)\\
&\leq \frac{1}{\kappa - \kappa_n}\Big( \frac{\omega_{\kappa_n, \ep}}{\kappa_n} - \frac{\omega_{\kappa, \ep}}{\kappa} \Big) + 6c\\
&\leq 8c,
\end{align*}
where we used~\eqref{eq:difference-quotient} to get the last line. Consequently,
\[
L_\ep(\gamma_n(t)) \leq 8 \kappa^2c,
\]
as asserted in (b).
\end{proof}

\subsection{Existence of non-trivial critical points for the perturbed functional}
\label{SS:pseudo-gradient flow}
\begin{lemm}\label{lemm:small-L-not-max}
Given $\kappa  > 0$, there exist positive constants $\eta_1, \eta_2 > 0$ and $\ep_0 < 1/2$, depending only on $\kappa$, such that if $\ep \in (0, \ep_0)$, $\gamma \in \cP_{\ep}$ and $t_0 \in [0, 1]$ are such that
\[
L_{\ep}(\gamma(t_0)) < \eta_1,
\]
then 
\[
L_{\kappa, \ep}(\gamma(t_0), f_{\gamma, t_0}) < \max_{t \in [0, 1]} L_{\kappa, \ep}(\gamma(t), f_{\gamma, t}) - \eta_2.
\]
\end{lemm}
\begin{proof}
We first note that there exists $\alpha_0 > 0$ such that $\max_{t \in [0, 1]}L(\gamma(t)) \geq \alpha_0$, for otherwise the induced map $h_\gamma$ would be homotopic to a constant, contradicting the definition of $\cP_\ep$. By Lemma~\ref{lemm:basic-estimates}(a), provided $\ep_0 < \frac{\alpha_0}{2 A_0}$, this means 
\[
\max_{t \in [0, 1]} L_{\ep}(\gamma(t)) \geq (\frac{\alpha_0}{2A_0})^{1 + \ep} \geq (\frac{\alpha_0}{2A_0})^2 =: \alpha_1.
\]
Thus, with $\eta_1 < \alpha < \alpha_1$ ($\eta_1$ and $\alpha$ to be determined), there must exist some $t' \neq t_0$ so that $L_{\ep}(\gamma(t'))  = \alpha$. Assuming, without loss of generality, that $t' > t_0$, we let $t_1 = \inf\{t > t_0 \ |\ L_{\ep}(\gamma(t)) \geq \alpha\}$. In particular 
\begin{equation}\label{eq:length-small}
L_\ep(\gamma(t)) \leq \alpha \text{ for all }t \in [t_0, t_1],
\end{equation}
with equality holding at $t = t_1$. Next we claim that
\begin{equation}\label{eq:area-difference}
|A(f_{\gamma, t_1}) - A(f_{\gamma, t_0})| \leq C \ep^2 + C\alpha^{\frac{1}{3}}\big(  L_{\ep}(\gamma(t_0)) + L_{\ep}(\gamma(t_1))\big),
\end{equation}
with $C$ independent of $\ep$. To see this, we choose arbitrary points $c_i \in \gamma(t_i)(S^1)$ ($i = 0, 1$) and note that since by Lemma~\ref{lemm:basic-estimates}(a) and~\eqref{eq:length-small},
\begin{equation}\label{eq:oscillation-bound}
\|\gamma(t_i) - c_i\|_{\infty} \leq L(\gamma(t_i)) \leq A_0 \big( \ep + \alpha^{\frac{1}{1 + \ep}}\big) \leq A_0(\ep_0 + \alpha^{\frac{2}{3}}),
\end{equation}
it makes sense to define the following extension of $\gamma(t_i)$ provided $\alpha, \ep_0$ are sufficiently small depending on $A_0$ and the constant $\delta_0$ from Lemma~\ref{lemm:area-properties}:
\[
h_i(s, \theta) = \Pi\big( s\gamma(t_i)(\theta) + (1 - s)c_i \big).
\]
A direct computation as in the proof of Lemma~\ref{lemm:basic-estimates}(c) shows that 
\begin{equation}\label{eq:cap-area-bound}
|A(h_i)| \leq C\| \gamma(t_i) - c_i\|_{\infty}L(\gamma(t_i)) \leq C\big( L(\gamma(t_i)) \big)^2,
\end{equation}
where $C$ depends only on $K := \| d\Pi\|_{\infty}$. Note also that $|\big( h_i(s, \cdot) \big)_{\theta} | \leq K|\gamma(t_i)_{\theta}|$, and thus 
\begin{align*}
L_{\ep}(h_i(s, \cdot)) + 2\pi \ep^{1 + \ep} &= \int_{S^1}\big( \ep^2 + |(h_i(s, \cdot))_{\theta}|^2 \big)^{\frac{1 + \ep}{2}} d\theta\\
 &\leq K^{1 + \ep}(L_\ep(\gamma(t_i)) + 2\pi \ep^{1 + \ep}) \leq K^2(\alpha + 2\pi \ep^{1 + \ep}), \text{ for all }s \in [0, 1].
\end{align*}
In other words,
\[
L_{\ep}(h_i(s, \cdot)) \leq K^2\alpha + 2\pi ( K^2 - 1) \ep_0 < K^2(\alpha + 2\pi\ep_0).
\]
Combining this with~\eqref{eq:length-small}, we see that, decreasing $\alpha, \ep_0$ further if necessary so that $(1 + K^2)(\alpha + 2\pi \ep_0) < \alpha_1$, the concatenation $h_0  - f_{\gamma, t_0} + f_{\gamma, t_1} - h_1$ induces a null-homotopic map from $S^2$ to itself. Consequently by~\eqref{eq:cap-area-bound}, Lemma~\ref{lemm:basic-estimates} and~\eqref{eq:length-small} we infer that 
\begin{align*}
|A(f_{\gamma, t_1}) - A(f_{\gamma, t_0})| &\leq |A(h_0)| + |A(h_1)|\\
& \leq C\big( (L_{\ep}(\gamma(t_0)))^{\frac{1}{1 + \ep}} + \ep \big)^2 + C\big( (L_{\ep}(\gamma(t_1)))^{\frac{1}{1 + \ep}} + \ep \big)^2\\
& \leq C\ep^2+ C\big( L_\ep(\gamma(t_0)) \big)^{\frac{2}{1 + \ep}} + C\big( L_\ep(\gamma(t_1)) \big)^{\frac{2}{1 + \ep}}\\
&\leq C\ep^2 + C\alpha^{\frac{1 - \ep}{1 + \ep}}\big(  L_{\ep}(\gamma(t_0)) + L_{\ep}(\gamma(t_1)) \big),
\end{align*}
which implies the claim since $\frac{1 - \ep}{1 + \ep} > \frac{1}{3}$ if $\ep \in (0, 1/2)$. We next use~\eqref{eq:area-difference} and the triangle inequality to compute
\begin{align*}
L_{\kappa, \ep}(\gamma(t_1), f_{\gamma, t_1}) - L_{\kappa, \ep}(\gamma(t_0), f_{\gamma, t_0}) &\geq L_{\ep}(\gamma(t_1)) - L_{\ep}(\gamma(t_0)) - \kappa |A(f_{\gamma, t_1}) - A(f_{\gamma, t_0})|\\
&\geq (1  - C\kappa\alpha^{\frac{1 }{3}})L_{\ep}(\gamma(t_1)) - (1  + C\kappa\alpha^{\frac{1}{3}})L_{\ep}(\gamma(t_0)) - C\kappa \ep^2\\
& > (1  - C\kappa\alpha^{\frac{1}{3}})\alpha - (1  + C\kappa\alpha^{\frac{1}{3}})\eta_1 - C\kappa \ep_0^2,
\end{align*}
where we used the fact that $L_{\ep}(\gamma(t_1)) = \alpha$ and $L_{\ep}(\gamma(t_0)) < \eta_1$ to get the last line. Upon requiring, in addition to the above thresholds on $\alpha$ and $\ep_0$, that $C\kappa\alpha^{\frac{1}{3}} < 1/2$, and then choosing $\eta_1$ and $\ep_0$ so that $(1  + C\kappa\alpha^{\frac{1}{3}})\eta_1 < \alpha/8$ and $C\kappa \ep_0^2 < \alpha/8$, we conclude the proof with $\eta_2 = \alpha/4$.
\end{proof}

In the pseudo-gradient flow argument below, we shall only deform the sweepouts where $L_{\kappa, \ep}$ is close to the min-max value $\omega_{\kappa, \ep}$. In such regions there is in fact a single well-defined reduction of $L_{\kappa. \ep}$. Specifically, fixing $\ep \in (0, 1/2)$, $\kappa > 0$ and $r <\frac{1}{4} \kappa \Area_g(S^2)$, we define
\[
\cN = \{u \in W^{1, 1 + \ep}(S^1; S^2)\ |\ |L_{\kappa, \ep}(u, f) - \omega_{\kappa, \ep}| < r \text{ for some }f\in \cE(u)\}.
\]
By Lemma~\ref{lemm:basic-estimates}(c) it's not hard to see that $\cN \cap \{L(u) < C\}$ is open for all $C < \infty$, and hence $\cN$ is open. We next define the reduction on $\cN$ mentioned above. Given $u \in \cN$, we choose $f \in \cE(u)$ such that $|L_{\kappa, \ep}(u, f) - \omega_{\kappa, \ep}| < r$ and set
\[
L^{\cN}_{\kappa, \ep}(u) = L_{\kappa, \ep}(u, f).
\]
Note that $L^{\cN}_{\kappa, \ep}(u)$ is well-defined since by Lemma~\ref{lemm:area-properties}(a) and our choice of $r$, any two such choices of extensions in $\cE(u)$ enclose the same area.

\begin{lemm}\label{lemm:local-reduct-patch}
$L^{\cN}_{\kappa, \ep}$ is a $C^{1}$-functional on $\cN$, and $\delta L^{\cN}_{\kappa, \ep} = \delta L_{\kappa, \ep}$ on $\cN$.
\end{lemm}
\begin{proof}
We will show that each $u_0 \in \cN$ has a simply-connected neighborhood $\cA$ on which $L^{\cN}_{\kappa, \ep}$ coincides with $L^{\cA}_{\kappa, \ep}$, which implies the result by Lemma~\ref{lemm:local-reduct}(b). To that end, let $f_0 \in \cE(u_0)$ be an extension such that $L^{\cN}_{\kappa, \ep}(u_0) = L_{\kappa, \ep}(u_0, f_0)$ and consider the local reduction $L^{\cA}_{\kappa, \ep}$ induced by $(u_0, f_0)$ on a simply-connected neighborhood $\cA$ of $u_0$. In particular $L^{\cA}_{\kappa, \ep}(u_0) = L^{\cN}_{\kappa, \ep}(u_0)$. Now by Lemma~\ref{lemm:basic-estimates}(c) and the openness of $\cN$, we may choose $\cA$ so that $\cA \subset \cN$ and 
\[
|L^{\cA}_{\kappa, \ep}(u) - L^{\cA}_{\kappa, \ep}(u_0)| < r \text{ for all }u \in \cA.
\]
Then for $u \in \cA$ we have 
\begin{align*}
|L^{\cA}_{\kappa, \ep}(u) - L^{\cN}_{\kappa, \ep}(u)| &\leq |L^{\cA}_{\kappa, \ep}(u) - L^{\cA}_{\kappa, \ep}(u_0)| + |L^{\cN}_{\kappa, \ep}(u_0) - \omega_{\kappa, \ep}| + | \omega_{\kappa, \ep} - L^{\cN}_{\kappa, \ep}(u) | \\
& < 3r < \kappa \Vol_g(S^2),
\end{align*}
which implies $L^{\cA}_{\kappa, \ep}(u) = L^{\cN}_{\kappa, \ep}(u)$ by Lemma~\ref{lemm:area-properties}(a).
\end{proof}

\begin{prop}\label{prop:extract}
Given $\kappa > 0, \ep \in (0, \ep_0)$, suppose for some $c> 0$ there exist sweepouts $\gamma_n \in \cP_{\ep}$ satisfying the conclusions of Proposition~\ref{prop:length-bound}. Then, passing to a subsequence if necessary, there exist $t_n \in [0, 1]$ such that 
\begin{enumerate}
\item[(a)] $|L_{\kappa, \ep}(\gamma_n(t_n), f_{\gamma_n, t_n}) - \omega_{\kappa, \ep}| \leq \frac{\kappa}{n}$.
\vskip 1mm
\item[(b)] $\gamma_n(t_n)$ converges strongly in $W^{1, 1 + \ep}(S^1; S^2)$ to a critical point $u$ of $L_{\kappa, \ep}$ with $L_{\ep}(u) \leq 8\kappa^2c$, and $L_{\kappa, \ep}(u, f) = \omega_{\kappa, \ep}$ for some $f \in \cE(u)$.
\vskip 1mm
\item[(c)] $\eta_1 \leq L_{\ep}(u)$. In particular $u$ is non-constant.
\end{enumerate}
\end{prop}
\begin{proof}
For brevity we write $\alpha_n = \kappa/n$ and $C_0 = 8\kappa^2 c$, and let 
\[
J_n = \{t \in [0, 1]\ |\ L_{\kappa, \ep}(\gamma_n(t), f_{\gamma_n, t}) > \omega_{\kappa, \ep} - \alpha_n\}.
\]
\[
I_n = \{t \in [0, 1]\ |\ L_{\kappa, \ep}(\gamma_n(t), f_{\gamma_n, t}) \geq \omega_{\kappa, \ep} - \alpha_n/2\}.
\]
With $r$ as chosen above Lemma~\ref{lemm:local-reduct-patch}, we first want to prove the following statement:

\vspace{0.5em}
\noindent  \textit{For all $\beta \in (0, r)$, there exists $n_0 \in \NN$ such that}
\begin{equation}\tag{$\ast$}\label{eq:statement-star}
\inf\{\|\delta L_{\kappa, \ep}(\gamma_n(t))\|\ |\ t \in J_n\} < \beta, \text{ for all }n \geq n_0.
\end{equation}

\vspace{0.5em}
Assume by contradiction that there exists some $\beta \in (0, r)$ and a subsequence which we do not relabel, such that for all $n$, we have 
\begin{equation}\label{eq:contradiction}
\|\delta L_{\kappa, \ep}(\gamma_n(t))\| \geq \beta \text{ for all }t \in J_n.
\end{equation}
Letting $\cN$ be defined as above, then for large enough $n$ we have $\gamma_n(t) \in \cN$ provided $t \in J_n$, in which case 
\begin{equation}\label{eq:agree}
L_{\kappa, \ep}(\gamma_n(t), f_{\gamma_n, t})  = L^{\cN}_{\kappa, \ep}(\gamma_n(t)).
\end{equation}
Next, following the proof of~\cite[Theorem 3.4]{StruweBook} we introduce the sets
\[
\cK= \{u \in W^{1, 1 + \ep}(S^1; S^2)\ |\ \delta L_{\kappa, \ep}(u) = 0, L_{\ep}(u) \leq C_0,\ L_{\kappa, \ep}(u, f) = \omega_{\kappa, \ep} \text{ for some }f \in \cE(u)\}.
\]
\[
\cU_{\delta} = \{u \in \cN\ |\ \|\delta L_{\kappa, \ep}(u)\| < \delta, L_{\ep}(u) < C_0 + \delta, |L^{\cN}_{\kappa, \ep}(u) - \omega_{\kappa, \ep}| < \delta  \}.
\]
\[
\cV_{\rho} = \{u \in \cN\ |\ \|u - v\|_{1, 1 + \ep} < \rho \text{ for some }v \in \cK\}.
\]
By Proposition~\ref{prop:PS} and Lemma~\ref{lemm:basic-estimates}(c) we see that $\cK \subset \cN$ is a compact set, and that both $\{\cU_\delta\}_{\delta > 0}$ and $\{\cV_{\rho}\}_{\rho>0}$ form fundamental systems of neighborhoods of $\cK$ (see~\cite[Lemma 2.3]{StruweBook}). We briefly explain the argument for $\{\cU_{\delta}\}_{\delta >0}$. Assume by contradiction that there exists a neighborhood $\cB$ of $\cK$ and a sequence $\delta_j \to 0$ such that for all $j$ we can find $u_j \in \cU_{\delta_j} \setminus \cB$. By Proposition~\ref{prop:PS}, $u_j$ has a subsequence, which we do not relabel, converging strongly in $W^{1, 1 + \ep}(S^1; S^2)$ to $u$ satisfying $\delta L_{\kappa, \ep}(u) = 0$ and $L_{\ep}(u) \leq C_0$. Next, for sufficiently large $j$, by Lemma~\ref{lemm:basic-estimates}(a)(c), the strong $W^{1, 1 + \ep}$-convergence of $u_j$ to $u$ and Sobolev embedding, we see that there exists $f_j \in \cE(u)$ such that  
\[
\lim_{j \to \infty}|L^{\cN}_{\kappa, \ep}(u_j) - L_{\kappa, \ep}(u, f_j)| = 0.
\]
Recalling the definition of $\cU_{\delta_j}$ and that $\delta_j \to 0$, we deduce that 
\[
\lim_{j \to \infty} |L_{\kappa, \ep}(u, f_j) - \omega_{\kappa, \ep}| = 0,
\]
and hence the sequence $L_{\kappa, \ep}(u, f_j)$ is eventually constantly equal to $\omega_{\kappa, \ep}$ by Lemma~\ref{lemm:area-properties}(a). This shows that $u \in \cK$, a contradiction since $u_j \notin \cB$ for all $j$.

Returning to the main line of argument, there exist $\rho, \mu$ sufficiently small so that 
\[
\cU_{\beta} \supset \cV_{2\rho} \supset \cV_{\rho} \supset \cU_{\mu}.
\]
This implies by~\eqref{eq:contradiction} that for $n$ large enough, the set $\gamma_n(J_n)$ is disjoint from $\cU_{\beta}$ and hence is separated from $\cU_{\mu}$ by at least a distance of $\rho$. That is,
\begin{equation}\label{eq:separated}
B_{\rho}(\gamma_n(J_n)) \cap \cU_{\mu} = \emptyset.
\end{equation}
On the other hand, if $n$ is so large the $\alpha_n < \mu/2$, it is not hard to see with the help of Lemma~\ref{lemm:basic-estimates}(a)(c), the upper bound in Proposition~\ref{prop:length-bound}(b) and~\eqref{eq:agree} that there exists $\rho' < \rho$, independent of $n$, such that,  
\begin{equation}\label{eq:contained}
B_{\rho'}(\gamma_n(J_n)) \subset \{u \in \cN\ |\ L_{\ep}(u) < C_0 + \mu,\ |L^{\cN}_{\kappa, \ep}(u) - \omega_{\kappa, \ep}| < \mu\}.
\end{equation}

Now since $L^\cN_{\kappa, \ep}$ is a $C^{1}$-functional on $\cN$, it possesses a pseudo-gradient vector field $X: \cN \setminus \cC \to W^{1, 1 + \ep}(S^1; \RR^N)$, where $\cC = \{u \in \cN\ |\ \delta L_{\kappa, \ep}(u) = 0 \}$ (see~\cite[Lemma 3.9]{StruweBook} or~\cite[p. 206]{Palais70}), with the property that 
\begin{enumerate}
\item[(p1)] $X(u) \in \cT_u$ for all $u \in \cN \setminus \cC$.
\item[(p2)] $\|X(u)\|_{1, 1 + \ep} < 2\min\{1, \|\delta L_{\kappa, \ep}(u)\|\}$.
\item[(p3)] $\langle \delta L_{\kappa, \ep}(u), X(u) \rangle < -\min\{1, \|\delta L_{\kappa, \ep}(u)\|\}\|\delta L_{\kappa, \ep}(u)\|$.
\end{enumerate}
Consider the flow generated by $X$, denoted 
\[
\Phi: \{(s, u)\ |\ u \in \cN \setminus \cC, s \in [0, T(u))\} \to \cN,
\]
where $T(u)$ is the maximal existence time for the integral curve starting at $u $. Then by property (p2) along with~\eqref{eq:separated} and~\eqref{eq:contained}, we see that for all $n$ large enough and $u \in \gamma_n(J_n)$, there holds $T(u) \geq \rho'/2$ and 
\[
\|\delta L_{k, \ep}(\Phi_s(u))\| \geq \mu,\text{ for } s \in [0, \rho'/2).
\]
Combining this with (p3), Proposition~\ref{prop:length-bound}(a) and~\eqref{eq:agree}, we see that following hold for all $t \in J_n$: First, 
\begin{equation}\label{eq:decrease}
L_{\kappa, \ep}^{\cN}\big(\Phi_{s}(\gamma_n(t))\big) \leq L_{\kappa, \ep}(\gamma_n(t), f_{\gamma_n, t}), \text{ for }s \in [0, \rho'/2);
\end{equation}
Second, for $n$ large enough,
\begin{equation}\label{eq:pushdown}
L_{\kappa, \ep}^{\cN}(\Phi_{\rho'/3}(\gamma_n(t))) \leq L_{\kappa, \ep}^{\cN}(\gamma_n(t))  - \frac{\rho' \mu^2}{3} < \omega_{\kappa, \ep} + \alpha_n - \frac{\rho' \mu^2}{3} < \omega_{\kappa, \ep} - \alpha_n/2.
\end{equation}

To continue, we take a continuous cut-off function $\zeta_n$ so that $\zeta_n(t) = 1$ on $I_n$ and $\zeta_n(t) = 0$ outside of a compact subset of $J_n$, and define, for $(s, t) \in [0, 1] \times [0, 1]$, 
\[
\Gamma_n(s, t) = \left\{
\begin{array}{cc}
\Phi\big(s\zeta_n(t)\rho'/3, \gamma_n(t)\big) & \text{, if }t \in J_n,\\
\gamma_n(t) & \text{, if }t \notin J_n.
\end{array}
\right.
\]
Letting $\widetilde{\gamma}_n = \Gamma_n(1, \cdot)$, then by Lemma~\ref{lemm:small-L-not-max}, eventually $0, 1$ lie outside of $J_n$ and hence are still mapped to constants by $\widetilde{\gamma}_n$. The continuity properties of the flow $\Phi$ then implies that $\widetilde{\gamma}_n \in \cP_{\ep}$. Moreover, we claim the following two properties: First of all, 
\begin{equation}\label{eq:agree-outside}
L_{\kappa, \ep}(\widetilde{\gamma}_n(t), f_{\widetilde{\gamma}_n, t}) = L_{\kappa, \ep}(\gamma_n(t), f_{\gamma_n, t}), \text{ for }t \notin J_n.
\end{equation}
Secondly, 
\begin{equation}\label{eq:flow-line-inside}
L_{\kappa, \ep}(\widetilde{\gamma}_n(t), f_{\widetilde{\gamma}_n, t}) = L^{\cN}_{\kappa, \ep}(\widetilde{\gamma}_n(t)), \text{ for }t \in J_n. 
\end{equation}
To see these, note that given $t \in [0, 1]$, we may use $\Gamma_n$ to construct a homotopy of extensions to show that 
\begin{equation}\label{eq:flow-replace-path}
L_{\kappa, \ep}(\widetilde{\gamma}_n(t), f_{\widetilde{\gamma}_n, t}) = L_{\kappa, \ep}(\widetilde{\gamma}_n(t), f_{\gamma_n, t} + \Gamma_n(\cdot, t)),
\end{equation}
which implies~\eqref{eq:agree-outside} by the definition of $\Gamma_n$ and the fact that $\gamma_n(t) = \widetilde{\gamma}_n(t)$ for $t \notin J_n$. As for~\eqref{eq:flow-line-inside}, note that by an argument similar to Lemma~\ref{lemm:min-max-properties}(a), for $t \in J_n$, the function
\[
s  \mapsto L_{\kappa, \ep}(\Gamma_{n}(s, t), f_{\gamma_n, t} + \Gamma_n(\cdot, t)|_{[0, s]})
\]
is continuous on $[0, 1]$. Hence, by Lemma~\ref{lemm:area-properties}(a) it differs from $s \mapsto L^{\cN}_{\kappa, \ep}(\Gamma_n(s, t))$ by a fixed integer multiple of $\kappa\cdot \Area_g(S^2)$. Since the two functions coincide at $s = 0$ by~\eqref{eq:agree}, they agree for $s \in [0, 1]$. This combined with~\eqref{eq:flow-replace-path} give~\eqref{eq:flow-line-inside}. 

By~\eqref{eq:agree-outside},~\eqref{eq:flow-line-inside},~\eqref{eq:decrease} and~\eqref{eq:pushdown} we see that
\[
L_{\kappa, \ep}(\widetilde{\gamma}_n(t), f_{\widetilde{\gamma}_n, t}) < \omega_{\kappa, \ep} - \frac{\alpha_n}{2}, \text{ for all }t \in I_n,\]
and that
\[
L_{\kappa, \ep}(\widetilde{\gamma}_n(t), f_{\widetilde{\gamma}_n, t}) \leq L_{\kappa, \ep}(\gamma_n(t), f_{\gamma_n, t}) \leq \omega_{\kappa, \ep} - \frac{\alpha_n}{2}, \text{ for }t \notin I_n.
\] 
Hence, we conclude that for $n$ sufficiently large there holds 
\[
\max_{t \in [0, 1]}L_{\kappa, \ep}(\widetilde{\gamma}_n(t), f_{\widetilde{\gamma}_n, t}) < \omega_{\kappa, \ep} - \frac{\alpha_n}{2},
\]
which contradicts the fact that $\widetilde{\gamma}_n \in \cP_{\ep}$, and hence the property~\eqref{eq:statement-star} must hold. Consequently there exists a subsequence of $\gamma_n$, which we do not relabel, and a sequence of times $t_n \in J_n$, such that $\gamma_n(t_n)$ has, by Proposition~\ref{prop:length-bound}(b) and Proposition~\ref{prop:PS}, a subsequence converging strongly in $W^{1, 1 + \ep}$ to $u$ with $\delta L_{\kappa, \ep}(u) = 0$ and $L_{\ep}(u) \leq C_0$. 

We are now ready to verify the conclusions of the Proposition. Part (a) is immediate from the definition of $J_n$. For part (b) it remains to find $f \in \cE(u)$ such that $L_{\kappa, \ep}(u, f) = \omega_{\kappa, \ep}$. Note that by the same reasoning as in the proof that $\{\cU_{\delta}\}_{\delta > 0}$ form a fundamental system of neighborhoods of $\cK$, for sufficiently large $n$ there exists $f_n \in \cE(u)$, with the property that
\[
|A(f_{\gamma_n, t_n}) - A(f_n)| \leq C \|\gamma_n(t_n) - u \|_{1, 1 + \ep} \big(2\ep +  L_\ep(\gamma_n(t_n))^{\frac{1}{1 + \ep}} + L_\ep(u)^{\frac{1}{1 + \ep}} \big) \to 0 \text{ as }n \to \infty.
\]
Consequently
\[
\lim_{n \to \infty}|L_{\kappa, \ep}(\gamma_n(t_n), f_{\gamma_n, t_n}) - L_{\kappa, \ep}(u, f_n)| = 0.
\]
Combining this with (a), which we just proved, and Lemma~\ref{lemm:area-properties}(a), we see that eventually $L_{\kappa, \ep}(u, f_n)$ is constantly equal to $\omega_{\kappa, \ep}$, yielding the desired extension of $u$. Finally, part (c) can be deduced from the strong $W^{1, 1 + \ep}$-convergence of  $\gamma_n(t_n)$ to $u$, together with Lemma~\ref{lemm:small-L-not-max}, the definition of $J_n$, and the fact that $\alpha_n \to 0$ as $n \to \infty$. The proof is complete.
\end{proof}

\section{Passage to the limit as $\ep \to 0$}
\label{S:pass to limit}
In this section we complete the proof of the main existence theorem. By Proposition~\ref{prop:struwe-trick}(c) and Proposition~\ref{prop:length-bound}, for almost every $\kappa$ there exists a sequence $\ep_j \to 0$ and a constant $c$ such that for all $j$, there exist sweepouts $\{\gamma_n\} \subset \cP_{\ep_j}$ to which we may apply Proposition~\ref{prop:extract} to extract a non-trivial critical point $u_j$ of $L_{\kappa, \ep_j}$ satisfying 
\begin{equation}\label{eq:upper-lower-bound}
\eta_1 \leq L_{\ep_j}(u_j) \leq 8\kappa^2 c.
\end{equation}
By Proposition~\ref{prop:smooth} we know that $u_j$ is smooth, and that $l_j := |u_j'|$ is constant. Consequently, by~\eqref{eq:upper-lower-bound} we have
\begin{equation}\label{eq:speed-bound}
\left(\frac{\eta_1}{2\pi}  + \ep_j^{1+\ep_j}\right)^{\frac{2}{1 + \ep_j}} - \ep_j^2 \leq l_j^2 \leq \left( \frac{4\kappa^2C}{\pi}  + \ep_j^{1+\ep_j}\right)^{\frac{2}{1 + \ep_j}} - \ep_j^2.
\end{equation}

The proposition below finishes the proof of Theorem~\ref{thm:main1}.
\begin{prop}\label{prop:convergence}
Passing to a subsequence if necessary, $u_j$ converge smoothly on $S^1$ to a non-trivial limit $u$ which, after reparametrization if necessary, satisfies $|u'| = 1$ and
\[
u'' = A_u(u', u') + \kappa Q_u(u').
\]
\end{prop}
\begin{proof}
Since each $u_j$ has constant speed $l_j$, the equation~\eqref{eq:Euler-Lagrange} becomes
\begin{equation}\label{eq:E-L-approximate}
\int_{S^1}\big[ (1 + \ep_j) u_j' \cdot \psi' + (1 + \ep_j) A_{u_j}(u_j', u_j') \cdot \psi + \kappa (\ep_j^2 + l_j^2)^{\frac{1 - \ep_j}{2}} Q_{u_j}(u_j') \cdot \psi \big]d\theta = 0,
\end{equation}
for all $\psi \in C^{1}(S^1; \RR^N)$. Since $l_j$ is uniformly bounded from above, we infer by bootstrapping that $(u_j)$ is uniformly bounded in $C^{k}(S^1; \RR^N)$ for all $k$. Hence we may extract a subsequence, which we do not relabel, converging smoothly on $S^1$ to a limit $u$ having constant speed $l = \lim_{j \to \infty}l_j$, which must lie in $(0, \infty)$ by~\eqref{eq:speed-bound}. Moreover, passing to the limit in~\eqref{eq:E-L-approximate}, we see that $u$ satisfies
\[
-u'' + A_u(u', u') + \kappa l \cdot Q_u(u') = 0.
\]
Now reparametrize $u$ by setting $v(s) = u(s/l)$ for $s \in [0, 2\pi l]$. Then $|v'| \equiv 1$, and 
\begin{align*}
v''(s) &= l^{-2}u''(s/l) \\
&= l^{-2}A_u(u', u')(s/l) + \kappa l^{-1} \cdot Q_u(u')(s/l) \\
&= A_v(v', v')(s) + \kappa \cdot Q_v(v')(s).
\end{align*}
\end{proof}

\appendix

\section{Proofs of some standard estimates} 
\label{sec:standard-estimates}
\begin{proof}[Proof of Lemma~\ref{lemm:coercive}.]
For~\eqref{eq:F-estimate1}, we let $y_t = t y_1 + (1-  t) y_0$, and apply the fundamental theorem of calculus to write
\begin{align*}
& \big((dF)_{y_1} - (dF)_{y_0}\big) \cdot (y_1 - y_0) \\
& = (1 + \ep)\Big( \int_{0}^1 \frac{y_1 - y_0}{(\ep^2 + |y_t|^2)^{\frac{1 - \ep}{2}}} - (1 - \ep)\frac{y_t \cdot (y_1 - y_0)}{(\ep^2 + |y_t|^2)^{\frac{3 - \ep}{2}}} y_t \  dt \Big) \cdot (y_1 - y_0)\\
&\geq (1 + \ep) \int_{0}^1 (\ep^2 + |y_t|^2)^{\frac{\ep - 3}{2}}\big( |y_1 - y_0|^2 (\ep^2 + |y_t|^2) - (1 - \ep) |y_t|^2 |y_1 - y_0|^2  \big)  dt \\
&\geq \ep(1 + \ep)\int_{0}^1 (\ep^2 + |y_t|^2)^{\frac{\ep - 1}{2}}|y_1 - y_0|^2 dt,
\end{align*}
which implies~\eqref{eq:F-estimate1} since $\ep < 1$ and $|y_t|^2 \leq 2|y_0|^2 + 2|y_1|^2$.

Throughout the proof of~\eqref{eq:F-estimate2}, the inequalities~\eqref{eq:fake-triangle} and~\eqref{eq:fake-triangle-2}, along with the fact that 
\[
|y_1| + |y_0| \leq C(|y_1|^2 + |y_0|^2)^{1/2},
\]
will be used frequently without further comment. To begin, we assume without loss of generality that $|y_1| \geq |y_0|$, and write
\begin{align*}
\frac{1}{1 + \ep}\big((dF)_{y_1} - (dF)_{y_0}\big) =\ & \big( (\ep^2 + |y_1|^2)^{\ep/2}  - (\ep^2 + |y_0|^2)^{\ep/2} \big)\frac{y_1}{(\ep^2 + |y_1|^2)^{1/2}}\\
& + (\ep^2 + |y_0|^2)^{\ep/2} \frac{y_1 - y_0}{(\ep^2 + |y_1|^2)^{1/2}} \\
& + (\ep^2 + |y_0|^2)^{\ep/2}\Big( \frac{1}{(\ep^2 + |y_1|^2)^{1/2}} - \frac{1}{(\ep^2 + |y_0|^2)^{1/2}} \Big)y_0\\
=:\ & I + II + III.
\end{align*}
Estimating $I$ is rather straightforward:
\begin{align*}
|I| & \leq ||y_1|^2 - |y_0|^2|^{\frac{\ep}{2}} \leq (|y_1| + |y_0|)^{\ep/2}|y_1 - y_0|^{\ep/2}\\
&\leq C(\ep^2 + |y_0|^2 + |y_1|^2)^{\ep/4}|y_1 - y_0|^{\ep/2}.
\end{align*}
As for $II$, we note that since $|y_1| \geq |y_0|$, we have 
\begin{equation}\label{eq:y1-geq-y0}
\ep^2  + |y_1|^2 \geq \frac{1}{2}(\ep^2 + |y_1|^2 + |y_0|^2).
\end{equation}
Hence
\begin{align*}
|II| &\leq C(\ep^2 + |y_0|^2 + |y_1|^2)^{\ep/2 - 1/2}|y_1 - y_0|^{1 - \ep/2} |y_1 - y_0|^{\ep/2}\\
&\leq C(\ep^2 + |y_0|^2 + |y_1|^2)^{\ep/4}|y_1 - y_0|^{\ep/2}.
\end{align*}
Finally, for $III$, using~\eqref{eq:y1-geq-y0} and the fact that $|y_0|(\ep^2 + |y_0|^2)^{-1/2} \leq 1$, we have
\begin{align*}
|III| & \leq (\ep^2 + |y_0|^2)^{\ep/2}\frac{|y_0|\big| (\ep^2 + |y_1|^2)^{1/2} - (\ep^2 + |y_0|^2)^{1/2}  \big|}{(\ep^2 + |y_0|^2)^{1/2}(\ep^2 + |y_1|^2)^{1/2}}\\
&\leq (\ep^2 + |y_0|^2 + |y_1|^2)^{\ep/2 - 1/2}(|y_1| + |y_0|)^{1/2} |y_1 - y_0|^{1/2 - \ep/2}|y_1  - y_0|^{\ep/2}\\
&\leq C(\ep^2 + |y_0|^2 + |y_1|^2)^{\ep/4}|y_1 - y_0|^{\ep/2}.
\end{align*}
Combining the estimates for $I$, $II$ and $III$ gives~\eqref{eq:F-estimate2}.
\end{proof}

\begin{proof}[Proof of Lemma~\ref{lemm:G-dL-relation}.]
The first inequality follows since for all $\psi \in \cT_u$ with $\|\psi\|_{2, 2} \leq 1$, we have
\[
\delta L_{\kappa, \ep}(u)(\psi) = G_{\kappa, \ep}(u)(\psi) \leq \|G_{\kappa, \ep}(u)\| \|\psi\|_{2, 2} \leq \|G_{\kappa, \ep}(u)\|.
\]
For the second inequality, note that by the smoothness of the nearest-point projection and the chain rule for weak derivatives, the composition $P_u: S^1 \to \RR^{N \times N}$ lies in $W^{1, 1 + \ep}$, with 
\begin{equation}\label{eq:projection-bound}
\|P_u\|_{1, 1 + \ep} \leq C (1 + \| u' \|_{1 + \ep}).
\end{equation}
Recall also the basic fact that if $f, h \in W^{1, 1 + \ep}(S^1; \RR^N)$, then so does $fh$, in which case
\begin{equation}\label{eq:algebra}
\|fh\|_{1, 1 + \ep} \leq C\|f\|_{1, 1 + \ep} \|h\|_{1, 1+ \ep}.
\end{equation}
By~\eqref{eq:projection-bound} and~\eqref{eq:algebra}, we see that for all $\psi \in W^{1, 1 + \ep}(S^1; \RR^N)$ with $\|\psi\|_{1, 1 + \ep} \leq 1$, the matrix-vector product $P_u(\psi)$ belongs to $\cT_u$, and satisfies $\| P_u (\psi) \|_{1, 1 + \ep} \leq C(1 + \| u'\|_{1 + \ep})$. Consequently, 
\begin{align*}
G_{\kappa, \ep}(u)(\psi) &= \delta L_{\kappa, \ep}(u)(P_u(\psi)) \leq \| \delta L_{\kappa, \ep}(u)\| \|P_u(\psi)\|_{1, 1 + \ep}\\
&\leq \| \delta L_{\kappa, \ep}(u)\| \cdot C(1 + \| u'\|_{1 + \ep}).
\end{align*}
This proves the second inequality in the statement.
\end{proof}
\bibliographystyle{amsplain}
\bibliography{cgc-existence-arxiv}
\end{document}